\documentclass[12pt]{amsart}
\usepackage{amsfonts,amssymb,amsmath,amsthm,bbm,xcolor,ulem,graphicx}

\def\a{\mathfrak{a}}

\def\R{{\bf R}}
\def\N{{\bf N}}
\def\C{{\bf C}}
\def\F{{\bf F}}

\def\to{\rightarrow}

\def\LL{\mathcal{L}}
\def\S{{\rm S}}

\def\Op{\hbox{\tiny Opdam}}
\def\Du{\hbox{\tiny Dunkl}}
\newcommand{\eqdef}{\mathrel{\mathop=}:}

\newtheorem{theorem}{Theorem}[section]
\newtheorem{lemma}[theorem]{Lemma}

\newtheorem{proposition}[theorem]{Proposition}
\newtheorem{corollary}[theorem]{Corollary}

\newtheorem{note}[theorem]{Notation}
\newtheorem{remark}[theorem]{Remark}
\newtheorem{conjecture}[theorem]{Conjecture}
\theoremstyle{definition}
\numberwithin{equation}{section}

\begin{document}

\begin{center}
{\bf 
Sharp estimates for the Opdam-Cherednik $W$-invariant heat kernel\\ for the root system $A_1$}\\
Piotr Graczyk\footnote{Laboratoire de math\'ematiques, Universit\'e d'Angers, France, piotr.graczyk@univ-angers.fr} and Patrice Sawyer\footnote{Department of Mathematics and Computer Science, Laurentian University, Canada,\\ psawyer@laurentian.ca}
\end{center}

\section*{abstract}
In this version of the paper, we note important gaps in the previous version.

\section*{Important note on the previous version of this paper}

We inform the readers of the previous version of this paper article later published as \cite{PGPSx}
that this article contains several mistakes or gaps.

We write $f\asymp g$ whenever $C_1\, g\leq f\leq C_2\, g$ for some fixed constants $C_1>0$, $C_2>0$. 

Let $p_t(r,s)$ be the Opdam-Cherednik $W$-invariant heat kernel and $H$ the proposed  bound for that kernel \emph{i.e.} such that $p_t\asymp H$.  The strategy of the paper consisted in using the Weak parabolic minimum principle for unbounded domains on a subdivision of the domain $(t,r)\in(0,\infty)\times (0,\infty)$.   The idea was to find functions $h^+\asymp H$ and $h^-\asymp H$ for each subregion and show that $D=\frac{\partial~}{\partial t}- \Delta_{\rm{Opdam}}$ satisfy $D(h^+-p_t)\geq 0$ (for the upper bound) and $D(p_t-h^-)\geq 0$ (for the lower bound).  We then ``glued'' the functions together in order to get a global estimate.

The first mistake we wish to point out is for the region $t<M$ for some fixed $M$.  In this case, the condition in the weak parabolic minimum principle that $h^+-p_t\geq 0$ on the boundary of the region is not satisfied.  Consequently the proposed upper/lower bound is incorrect.

Another gap  is in the proof in the case $t>1$, where we did not ``glue'' the domains $r<s$ and $r>s$.

An error on page 203 of \cite{PGPSx} consisted in showing $Dh^+$ and $Dh^-$ are $O(t^{-2})$ whereas we need $Dh^+/h^+$ and $Dh^-/h^-$ to be both $O(t^{-2})$.

Indeed, if we have $\frac{Dh^+}{h^+}\geq -\frac{C}{t^2}$ with $C>0$ then
\begin{align*}
\frac{D(e^{-C/t}h^+)}{h^+}&=\frac{e^{-C/t}(Dh^++(C/t^2)h^+)}{h^+}=e^{-C/t}\left(\frac{Dh^+}{h^+}+\frac{C}{t^2}\right)\geq 0
\end{align*}
while if we only have $Dh^+\geq -\frac{C}{t^2}$ with $C>0$ as we have shown in the paper, then
\begin{align*}
D(e^{-C/t}h^+)&=e^{-C/t}\left(Dh^++(C/t^2)h^+\right)\geq e^{-C/t}\frac{C}{t^2}\left(-1+h^+\right)
\end{align*}
which may not be positive.

We thank Pawe\l\ Plewa \ for pointing out these gaps and errors to us (private communication and  \cite{Plewa}).
In \cite{Plewa}, Pawe\l\  Plewa proposes 
 sharp bounds for the $W$-invariant trigonometric Dunkl heat kernel for the root system $BC_1$ (which includes the $A_1$ case) and a proof of them.

\section*{abstract}
For the first time,
 the  estimates of the non-centered  Weyl-group invariant heat kernel  
for  curved Riemannian spaces  and for  Opdam-Cherednik Laplacians are studied systematically. 
 We prove sharp estimates for the root system $A_1$ with arbitrary multiplicity $k>0$.  Sharp estimates of
Opdam-Cherednik's  radial heat kernel are conjectured for any root system.

\section{Introduction}

In Opdam-Cherednik harmonic and stochastic analysis, the ``curved'' counterpart of Dunkl analysis for a root system $\Sigma$ in the underlying space $\a=\R^n$, a crucial role is played by the Opdam-Cherednik heat kernel $p_{\Op}(X,Y,t)$ and its Weyl-invariant version $h_{\Op}(X,Y,t)$.
Finding good estimates of  these heat kernels is therefore important. 

The novelty of this  paper
is the systematic study  of the  estimates of the non-centered radial (i.e. Weyl group invariant) heat kernel, denoted 
$h_{\Op}(X,Y,t)$, for  curved Riemannian spaces  and for an Opdam-Cherednik Laplacian (this is new except in \cite{PGPS1} for the complex $A$ case, see formula \eqref{est:complex} below).
We prove sharp estimates of $h_{\Op}(X,Y,t)$ for the   root system $A_1$ for arbitrary multiplicities $k>0$.  A conjecture on  sharp estimates of Opdam-Cherednik's  radial heat kernel $h_{\Op}(X,Y,t)$  is formulated  for any given root system, cf.{} Conjecture  \ref{conjecture}.

In the centered case $Y =0$, exact estimates of $h_{\Op}(X,0,t)$ were obtained  for all Riemannian symmetric spaces  by Anker and Ostellari \cite{AO} and, for Opdam-Cherednik's  radial heat kernel, by Schapira \cite{Schapira}. 
In \cite[Section 5.2, page 25]{Schapira}, the difficulty of getting sharp estimates
for the heat kernel in three variables $X$, $Y$, $t$ is pointed out and the result on $h_{\Op}(X,0,t)$ is commented as only a partial result. 


\subsection{Basics on Opdam-Cherednik operators and heat kernel}
$\;$\\

For a good introduction on Opdam-Cherednik analysis, the reader should consider
the lecture notes \cite{OpdamLN} by Opdam, \cite{Anker} by Anker and 
the papers \cite{AAS, Opdam}. The Opdam-Cherednik heat kernel is defined and studied by Schapira  \cite{ Schapira}.

 We provide here some details and notations on Opdam-Cherednik analysis.  For every root $\alpha\in \Sigma$, let $\sigma_\alpha(X)=X-2\,\frac{\alpha(X)}{\langle \alpha,\alpha\rangle}\,\alpha$. 
The Weyl group $W$ associated to the root system is generated by the reflection maps $\sigma_\alpha$. One chooses positive roots $\Sigma^+$ and one defines
the positive Weyl chamber $\a^+=\{X\in\a\colon \alpha(X)>0~\forall X\in\Sigma^+\}$. 
Note that in this paper   $n=\dim(\hbox{lin} \Sigma)$.  In the symmetric space setting, $n$ denotes the rank.

A function $k: \Sigma \to [0,\infty)$ is called a multiplicity function if it is invariant under the action of $W$ on $\Sigma$.
\begin{align*}
T_\xi\,f(X)&=\partial_\xi\,f(X)+\sum_{\alpha\in \Sigma^+}\,k(\alpha)\,\alpha(\xi)
\,\frac{f(X)-f(\sigma_\alpha\,X)}{1-e^{-\alpha(X)}}-\rho(k)(\xi)\,f(X),
\end{align*}
where $\rho(k)= \frac{1}{2}\,\sum_{\alpha\in \Sigma_+} k(\alpha)\alpha$.
The operators $T_\xi$'s, $\xi\in\a$, form a commutative family.

For fixed $ Y\in\a$, the  kernel $G_k(\cdot,Y)$ is  the only real-analytic solution to the system
\begin{align*}
\left.T_\xi(k)\right\vert_X\,G_k(X,Y)
=\langle \xi,Y\rangle\, G_k(X,Y),~\forall\xi\in\a
\end{align*}
with $G_k(0,Y)=1$.  In fact, $G_k$ extends to a holomorphic function  on $(\a+i\,U)\times \C^n$ where $U$ is a neighbourhood of 0 (refer to \cite[Th.{} 3.15]{Opdam}).
Its $W$-invariant version $G^W_k(X,\lambda)$ is called a Heckman-Opdam hypergeometric function.

The Heckman-Opdam  hypergeometric functions related to root systems are the extension of the spherical functions for noncompact symmetric spaces $\phi_\lambda$ to arbitrary positive multiplicities.
In this paper, we use the latter terminology and notation.  We have
\begin{align*}
\phi_\lambda(X) =G^W_k(X,\lambda)=\frac{1}{|W|}\sum_{w\in W}\,G_k(w\cdot X,\lambda).
\end{align*}

Let $\omega_k(X):=
\prod_{\alpha\in \Sigma^+}\,\vert2\,\sinh\frac{\langle \alpha,X\rangle}{2}\vert^{2\,k(\alpha)}
$
be the Opdam-Cherednik weight function  on $\a$ (c.f.{} \cite[p.{} 40]{Anker}).
Recall that the Opdam-Cherednik transform of a $W$-invariant function $f$ on $\a$
\begin{align*}
\hat f(\lambda):= c_k^{-1} \int_{\a} f(x) \phi_{-i\lambda}(X)\omega_k(X)dX,\qquad \lambda\in \a^*,
\end{align*}
plays the role of the spherical Fourier transform in the $W$-invariant Opdam-Cherednik analysis.

\subsubsection{Modified Laplacian and Heat operator}

Let
${\bf e}_i$, $i=1,\ldots,n$ be the canonical basis of $\a$.     Let
 $\LL=\sum_{i=1}^n T^2_{{\bf e}_i}$ be the Opdam-Cherednik Laplacian
on $C^{2}(\a)$.   We define the modified Opdam-Cherednik Laplacian by
\begin{align*}
L_{\Op}=  \LL-|\rho|^2.
\end{align*}
Note that in \cite[ Section 5.1, page 242]{Schapira},
a factor $1/2$ is used in front of $\LL-|\rho|^2$  in the definition of the modified Opdam-Cherednik Laplacian.
We keep to the convention without the factor $1/2$, following Helgason \cite{Hel} on Riemannian symmetric spaces,
Anker et al.{} \cite{AnkerDH} and Roesler \cite{Roesler} in the Dunkl setting.
In practice, formulas for the Opdam-Cherednik heat kernel  from \cite{Schapira} correspond to those
of this paper with $t/2$ replaced by $t$.

The Opdam-Cherednik heat operator $D_{\Op}$ is defined on $C^{2,1}(\a\times\R)$
 by
\begin{align*}
D_{\Op}= {\partial\over\partial t} -L_{\Op}.
\end{align*}
 
We will denote with a superscript $W$ the $W$-invariant variants of the operators: $\LL^W$, $L^W_{\Op}$ and $D^W_{\Op}$.

\subsubsection{Heat kernel}

The Opdam-Cherednik  heat kernel $p_{\Op}(X,Y,t)$  is defined by Schapira in \cite[Definition 5.1]{Schapira}. For all $X,Y\in\a$ and $t>0$
\begin{align*}
p_{\Op}(X,Y,t)= \int_{\a }e^{ -(|\lambda|^2+|\rho|^2)\,t}
G_k(X,i\lambda) G_k(-Y,i\lambda) d\nu(\lambda).
\end{align*}
where $\nu$ is the asymmetric Plancherel measure on $\a$, see e.g. \cite[p.18]{Schapira}

The $W$-invariant heat kernel is defined by
\begin{align*}
h_{\Op}(X,Y,t)=\frac1{|W|} \sum_{w\in W} p(X,wY,t).
\end{align*}

Consequently, we have
\begin{align}\label{Plancherel}
h_{\Op}(X,Y,t)= \int_\a e^{ -(|\lambda|^2+|\rho|^2)\,t}
\phi_{i\lambda}(X) \phi_{i\lambda}(-Y) d\nu'(\lambda),
\end{align}
where $\nu'$ is the symmetric Plancherel measure, called also Harish-Chandra measure,
see \cite[p.21]{Schapira}.

Note that throughout this paper we use the notation $h$  for $W$-invariant heat kernels and $p$ for the general  heat kernels, both in the Opdam-Cherednik and the Dunkl settings.

The estimates  of the functions $G_k(X,i\lambda) $ and $\phi_{i\lambda}(X)$ proven in \cite{Schapira} guarantee the convergence of the integrals defining the kernels $p_{\Op}$ and $h_{\Op}$.

The kernel $h_{\Op}(X,Y,t)$ is the fundamental solution of the $D^W$-Cauchy problem \cite{Schapira}.

\subsection{Dunkl case}
In the papers \cite{ PGPS1,PGPS0}, we studied the $W$-invariant heat kernel on flat symmetric spaces 
and on  $\R^n$ in the rational Dunkl setting for the root systems of type $A$.

In \cite{PGPS0}, we proved  exact estimates for the $W$-invariant Dunkl heat kernel $h_{\Du}(X,Y,t)$, for the root systems  $A_n$ in $\R^n$ with arbitrary positive multiplicities:
\begin{align}\label{FLAT}
h_{\Du}(X,Y,t)\asymp  \frac{t^{-n/2}\,e^{-|X-Y|^2/(4\,t)}}{\prod_{\alpha>0}\,(t+\alpha(X)\,\alpha(Y))^{k}}
\qquad (X,Y\in \a^+, t>0).
\end{align}

 Anker, Dziuba\'nski and Hejna \cite{AnkerDH, DH} studied the estimates of the non $W$-invariant Dunkl heat kernel $p_{\Du}(X,Y,t)$  for any root system but they did not get the same lower and upper estimates. 

\subsection{Complex curved case}
In the complex case $k=1$,  we proved in \cite[Corollary 13]{PGPS1} the following
sharp estimate of the  heat kernel  $h(X,Y,t)$ for the curved irreducible Riemannian
symmetric space of type $A_n$ in the complex case: 
\begin{align}\label{est:complex}
h(X,Y,t)&\asymp t^{-\frac{n}{2}} \,e^{\frac{-|X-Y|^2}{4t}} e^{-|\rho|^2 t} \,e^{-\rho(X+Y)}
\,\prod_{\alpha>0}\,\frac{(1+\alpha(X))\,(1+\alpha(Y))}{t+\alpha(X)\,\alpha(Y)}.
\end{align}

In particular, for the 3-dimensional real  hyperbolic space ${\rm H}^3(\R)$, we have $n=k=\rho=1$ and the formula \eqref{est:complex} implies that 
\begin{align}\label{est:complexhyperb}
h(X,Y,t)&\asymp  t^{-\frac{1}{2}} \,e^{\frac{-(X-Y)^2}{4t}} e^{- t} \,e^{-(X+Y)} \,\frac{(1+X)\,(1+Y)}{t+XY}
\qquad\quad (X,Y\ge 0, t>0).
\end{align}

\subsection{Conjecture for sharp estimates of $h_{\Op}(X,Y,t)$}

We conjecture that in the general  Opdam-Cherednik setting, we have the following sharp estimate.

\begin{conjecture}\label{conjecture}

\begin{align*}
h_{\Op}(X,Y,t)\asymp t^{-\frac{n}{2}} \,e^{\frac{-|X-Y|^2}{4t}} e^{-|\rho|^2 t} \,e^{-\rho(X+Y)}
\,\prod_{\alpha\in \Sigma^{++} }&
\,(1+\alpha(X))\,(1+\alpha(Y))
\\&\nonumber
\,
\times\frac{(t+1+\alpha(X+Y))^{k(\alpha)+ k(2\alpha)-1}}
{(t+\alpha(X)\,\alpha(Y))^{k(\alpha)+k(2\alpha)}}
\end{align*}
where $\Sigma^{++}$ is the set of positive indivisible roots.
\end{conjecture}

This Conjecture is compatible with
the following sharp global estimate for $h_{\Op}(X,0,t)$ proved by
Schapira \cite[Theorem 5.2]{Schapira} in the
general Opdam--Cherednik setting.
Here $\gamma=\sum_{\alpha\in \Sigma^+} k(\alpha)$ and $t/2$ in  \cite[Theorem 5.2]{Schapira} is replaced by $t$:
\begin{align}\label{est:AO}
h_{\Op}(X,0,t)&\asymp t^{-\gamma-\frac{n}{2}} \,e^{\frac{-|X|^2}{4t}}
e^{-|\rho|^2 t} \,e^{-\rho(X)}
\,\prod_{\alpha\in \Sigma^{++} }\,(1+\alpha(X))
\,(t+1+\alpha(X))^{k(\alpha)+k(2\alpha)-1},
\end{align}
This result contains an earlier estimate of Anker and Ostellari
\cite{AO} on the curved symmetric spaces $X$ ($k=\frac{1}{2}$, 1, 2) 
and its proof is based on the proof of Anker and Ostellari.

Conjecture \ref{conjecture} is also compatible with an asymptotic of $p_{\Op}(x,\sqrt{t}y,t)$ when $t\to\infty$ proven
in  \cite[Proposition 5.3]{Schapira}.

\begin{remark}
One  might think that $h_{\Op}(X,Y,t)$ is a function of the  distance $d(X,Y)$.
This is easily shown to be false even in the geometric (\emph{i.e.} symmetric space) case.

Consider the hyperbolic space  ${\rm H}^3(\R)\equiv SL(2,C)/SU(2)$.  In this case, by the explicit formula for the spherical functions (\cite
{Hel}) and the relation between the  spherical functions  and the $W$-invariant heat kernel (\cite{deJeu}, \cite{Roesler}), see also \cite[Remark 2.9]{PGPSMN} for a simple direct proof, we have
\begin{align*}
h_{\Op}(X,Y,t)=\frac{t^{-1/2} }{2 \delta(X)^{1/2} \delta(Y)^{1/2}}
(e^{-(X-Y)^2/(4t)}  -e^{-(X+Y)^2/(4t)}), \qquad  (
X,Y\ge 0, t>0).
\end{align*}

If we had $h_{\Op}(X,Y,t)=g_t(d(X,Y))$, then $h_{\Op}(X,X,t)=g_t(0)$  would be independent of $X$, but
\begin{align*}
h_{\Op}(X,X,t)=\frac{t^{-1/2} }{2 \delta(X)}  (1-e^{-X^2/t})
\end{align*}
which depends on $X$ and contradicts the assumption.
\end{remark}

In this paper, we will prove Conjecture \ref{conjecture} in the $A_1$ case.

\begin{theorem}\label{Main}
Conjecture \ref{conjecture} holds in the $A_1$ with arbitrary multiplicity $k>0$, namely
\begin{align}\label{estim:conjecture}
h_{\Op}(r,s,t)&\asymp t^{-\frac{1}{2}} \,e^{\frac{-|r-s|^2}{4t}} e^{-k^2 t} \,e^{-k\,(r+s)}
\,(1+r)\,(1+s) \,\frac{(t+1+r+s)^{k-1}}{(t+r\,s)^k}\\
&\asymp t^{-\frac{1}{2}} \,e^{\frac{-|r-s|^2}{4t}} e^{-k^2 t} \,\phi_0(r)\,\phi_0(s)
\,\frac{(t+1+r+s)^{k-1}}{(t+r\,s)^k}
\end{align}
for $t>0$, $r\geq 0$ and $s\geq0$.
\end{theorem}

\begin{remark}
The estimate $\phi_0(X)\asymp e^{-\rho(X)}\,\prod_{\alpha\in \Sigma^+}\,(1+\alpha(X))$ is known in its full generality (\cite[Theorem 3.1]{Schapira}).
\end{remark}

\section{Auxiliary results}

The results gathered in this chapter will be useful in the proof of the main theorem.

\subsection{Parabolic minimum principle }

\begin{proposition}[Weak parabolic minimum principle for unbounded domains]\label{Ao}
Let $\Omega\subseteq \R^d$  be a domain (i.e.{} a connected open subset). Let $t_0 < t_1$ be two real numbers.
Let $D=D_{\Du} $  or $D=D_{\Op} $.

Assume that $u\in C^{2,1}(\Omega\times (t_0,t_1])\cap C^0(\overline{\Omega}\times[t_0,t_1])$ satisfies
\begin{itemize}
\item[(i)] $D\,u\geq 0$ in $\Omega\times (t_0,t_1]$ (supersolution),
\item[(ii)] $u\geq 0$ on the boundary component $\Omega\times\{t_0\}$ and $\partial\,\Omega\times [t_0,t_1]$.
\end{itemize}
Then $u\geq 0$ throughout $\overline{\Omega}\times [t_0,t_1]$.
\end{proposition}
\begin{proof}
For the Dunkl case, see R\"osler \cite{Roesler}.
For the Opdam-Cherednik case, see Schapira \cite[p.26, proof of Theorem 5.2]{Schapira}.  
Weak parabolic minimum principle holds in these cases
for unbounded domains since the heat kernel vanishes at infinity, cf. Corollary \ref{cor:heat_infty}.
For general parabolic minimum principles for unbounded domains see DiBenedetto \cite[pp.{} 142-146]{Diben}.
\end{proof}

\subsection{Estimates of  the Legendre function $\phi_0$}

We will use the following properties
of the  Legendre function $\phi_0$. 
An analogous   result for $e^\rho\,\phi_0$ in place of $\delta^{1/2} \phi_0$
is proven  in \cite[Theorem 3.3]{Schapira}.
However,   $e^\rho\not=\delta^{1/2}$ although they are
equivalent at infinity.

\begin{lemma}\label{G}
We consider the Legendre function $\phi_0$
in the rank one $A_1$ case.  Let 
\begin{align*}
G(r)=r\,\frac{d~}{dr}\,\log\left(\delta^{1/2}\,\phi_0\right)(r).
\end{align*}

Then $G(0)=k$, $G(r)=1+K(r)/(1+r)$ where $|K(r)|\leq K_0$, $r\geq 0$,  for a constant $K_0>0$ and $G(r)>0$ for all $r> 0$.  In addition, there exists $K>0$ such that $|\phi_0'(r)|\leq K$ for all $r$.
\end{lemma}

\begin{proof}
The fact that $G(r)=1+K(r)/(1+r)$ with $K(r)$ bounded
is a consequence of \cite[Theorem 3.3]{Schapira}.
We have (refer to \cite{Sawyer})
\begin{align*}
\phi_0(r)&=\frac{\Gamma  \left(2 k \right)}{\Gamma  \left(k \right)^{2}}\, \sinh^{1-2 k}r\, \int_{-\frac{r}{2}}^{\frac{r}{2}}\left(\sinh  \left(\frac{r}{2}-u \right) \sinh  \left(u +\frac{r}{2}\right)\right)^{k -1}d u \\
&=\frac{\Gamma  \left(2 k \right)}{\Gamma  \left(k \right)^{2}}\,\int_0^1\,(e^r\,\beta+e^{-r}\,(1-\beta))^{-k}\,(\beta\,(1-\beta))^{k-1}\,d\beta
\end{align*}
using the change of variable $\beta=(e^{2\,u}-e^{-r})/(e^r-e^{-r})$.
Since  
\begin{align*}
G(r)&=r\,\frac{\frac{d~}{dr}\,\left(\delta^{1/2}\,\phi_0\right)(r)}{\delta^{1/2}(r)\,\phi_0(r)}
&=r\,\frac{\frac{d~}{dr}\,\left(\sinh^kr\,\phi_0(r)\right)}{\sinh^kr\,\phi_0(r)}
&=\frac{r}{\sinh r}\,\frac{k\,\cosh r\,\phi_0(r)+\sinh r\,\phi_0'(r)}{\phi_0(r)}
\end{align*}
the fact that $G(0)=k$ stems from the last equality, 
while the positivity $G(r)>0$ for all $r>0$ stems from the fact that
for $r>0$ we have
\begin{align*}
\frac{d~}{dr}\,\left(\sinh^k r\,\phi_0\right)(r)
&=\frac{d~}{dr}\,\sinh^kr\,
\left(\frac{\Gamma  \left(2 k \right)}{\Gamma  \left(k \right)^{2}}
\,\int_0^1\,(e^r\,\beta+e^{-r}\,(1-\beta))^{-k}\,(\beta\,(1-\beta))^{k-1}\,d\beta\right)\\
&=\frac{\Gamma  \left(2 k \right)}{\Gamma  \left(k \right)^{2}}\,\int_0^1\,\frac{d~}{dr}\,\left(\sinh^kr\,(e^r\,\beta+e^{-r}\,(1-\beta))^{-k}\right)\,(\beta\,(1-\beta))^{k-1}\,d\beta\\
&=\frac{\Gamma  \left(2 k \right)}{\Gamma  \left(k \right)^{2}}\,\int_0^1\,(e^r\,\beta+e^{-r}\,(1-\beta))^{-k-1}\,\sinh^{k-1}r
\,(\beta\,(1-\beta))^{k-1}\,d\beta>0.
\end{align*}

Since
\begin{align*}
\phi_0'(r)&=\frac{\phi_0(r) \left(G(r) \sinh  r\,- k r\,\cosh  r  \right)}{r\,\sinh  r},
\end{align*}
the rest follows.
\end{proof}

We give here a crude estimate for the $W$-invariant Opdam-Cherednik heat kernel.
\begin{lemma}\label{crude}
\begin{align*}
h_{\Op}(X,Y,t)\leq  \phi_0(X)\,\phi_0(-Y)\,\int_{\a}\,e^{-(|\lambda|^2+|\rho|^2)\,t}\,d\nu'(\lambda).
\end{align*}
\end{lemma}

\begin{proof}  By \eqref{Plancherel} we have
$
h_{\Op}(X,Y,t)
\leq  \displaystyle\int_{\a}\,e^{-(|\lambda|^2+|\rho|^2)\,t}\,|\phi_{i\lambda}(X)|\,|\phi_{i\lambda}(-Y)|
\,d\nu'(\lambda).
$
\end{proof}
\begin{corollary}\label{cor:heat_infty}
For each fixed $Y\in\a^+$ and $t>0$, we have  $
\lim_{|X|\to \infty} h_{\Op}(X,Y,t)=0$.
Moreover, $ h_{\Op}(X,Y,t)=O(e^{-|\rho|^2\,t})$ independently of $X$ and $Y$.
\end{corollary}

\subsection{A covering lemma}\label{sec:cover}

\begin{lemma}\label{cover}
 Let $S_a=\{t\colon s^a/2\leq t\leq 2\,s^a\}\subset\{t\colon s^a/4\leq t\leq 2\,s^a\}=T_a$.  Suppose $1\leq s\leq t\leq s^2$.  If $
 s\in (4^{2^{N-1} }, 4^{2^N} ]$ then
\begin{align*}
t\in \bigcup_{r=2^N}^{2^{N+1}}\,S_{r/2^N}.
\end{align*}
For $N$ fixed and $1\leq s\leq t\leq s^2$, $t$ will belong to no more than three of $S_{r/2^N}\cup T_{(r+1)/2^N}=[s^{r/2^N}/2,2\,s^{(r+1)/2^N}]$, $2^N\leq r\leq 2^{N+1}$.
\end{lemma}

\begin{proof}
Note that $s\leq 4^{2^N}$ is equivalent to $s^{r/2^N}/2\leq 2\cdot s^{(r-1)/2^N}$.
This means that consecutive intervals $S_{r/2^N}$ intersect 2 by 2, so that the union $\bigcup_{r=2^N}^{2^{N+1}}\,S_{r/2^N}$ is an interval containing 
$S_1$ and $S_2$ and, in particular, $s/2$ and $2\,s^2$. 

On the other hand, suppose $t\in(S_{u/2^N}\cup T_{(u+1)/2^N})\cap(S_{v/2^N}\cup T_{(v+1)/2^N})$ with $u\leq v$. That means that $2\,s^{(u+1)/2^N}\geq s^{v/2^N}/2$.
This is equivalent to $4\geq s^{(v-u-1)/2^N}$. If $v-u-1\geq 2$ then  $4\geq s^{2/2^N}>(4^{2^{N-1}} )^{1/2^{N-1}}=4$ when $s\in (4^{2^{N-1} }, 4^{2^N} ]$ which is absurd.  We must have $-1\leq v-u-1\leq 1$ \emph{i.e.} $v=u$, $v=u+1$ or $v=u+2$.
The result follows.
\end{proof}

\section{Proof of Theorem \ref{Main} }

For the root system $A_1$, we have $k(\alpha)=k$ and $k(2\,\alpha)=0$.
From now on, we use the one-dimensional notation $r$, $s\in \R^+=[0,\infty[$ in place of $X,Y$.
We have $D^W_{\Op}=\frac{\partial~}{\partial t}-L^W_{\Op}$, where 
\begin{align*}
L^W_{\Op}=\frac{d^2~}{dr^2}+2\,k\,\coth r\,\frac{d~}{dr}.
\end{align*}
In the sequel, for simplicity, we omit the superscript $W$ for  radial operators, i.e. we write $D$ for $D^W$.\\

The proof is based on the parabolic minimum principle (Proposition \ref{Ao}), cf.  Anker and Ostellari \cite{AO}.

\subsection{Plan of the proof}
 By symmetry, we can assume $r\geq s$.   We will make this hypothesis throughout
the proof.  Denote $\R^{2+}=\{(r,s)\,|\, r\geq s\geq 0\}.$\\

 We cut the space $\R^{2+}\times\R^+$ into  regions:  a ``lower'' Region 0
 (i.e. $t\le M$) and, for $t\ge M$, the ``upper'' sub-regions A, B, C, D.
 The region D will be divided into further subregions. Call $U$ the
 union of  the "upper" sub-regions A,B,C,D.
 
 The parabolic minimum principle (Proposition \ref{Ao})
will be applied two times:
\begin{itemize}
\item[(i)] on the lower Region 0

\item[(ii)] on the upper Region  $U$ equal to the union of  subregions A,B,C,D.\\
\end{itemize}

 Denote by $E(r,s,t)$ the right-hand side of the estimate formula \eqref{estim:conjecture}
from Theorem \ref{Main}:
$$
E(r,s,t) \eqdef t^{-\frac{1}{2}} \,e^{\frac{-|r-s|^2}{4t}-k^2 t-k\,(r+s)}
\,(1+r)\,(1+s) \,\frac{(t+1+r+s)^{k-1}}{(t+r\,s)^k}
$$

In order to apply the parabolic minimum principle,  we will construct auxiliary functions $h_{{\mathcal R},c}$ for each
of two regions ${\mathcal R}$ (equal to Region 0 or $U$) and real $c$, such that 
$h_{{\mathcal R},c}\asymp E$ on ${\mathcal R}$ and, for well chosen $c$, we have $D_{\Op}h_{{\mathcal R},c}\ge 0$  on ${\mathcal R}$
or $D_{\Op}h_{{\mathcal R},c}\le 0$  on ${\mathcal R}$.

The parabolic minimum principle will be applied to the function $u=h_{{\mathcal R},c}-
h_{\Op}$. The conclusive inequality  will give 
one inequality of the conjectured  estimate $h_{{\mathcal R},c}\asymp
h_{\Op}$.\\

On the lower Region 0, the kernel $h_{\Du}$ will be used to define a smooth
function $h_{0,c}$
such that $h_{0,c}\asymp E$ on  Region 0
and $D_{\Op} h_{0,c}\ge 0$
for $c$ big enough. To finish, we will interchange the roles of   $h_{\Du}$ and $h_{\Op}$.\\

On every upper subregion ${\mathcal S}\subset U$ and for every
 $c\in\R$,
 we will construct a smooth function $h_{{\mathcal S},c}$ such that
 $
 h_{{\mathcal S},c}\asymp E$ on $ {\mathcal S}$  and
there exist $M_1$, $M_2>0$  such that
\begin{eqnarray*}
 D_{\Op} h_{{\mathcal S},c}(r,s,t) \ge 0\qquad \text{ for\ all}\ c\ge M_1 \\
  D_{\Op} h_{{\mathcal S},c}(r,s,t) \le 0\qquad \text{ for\ all}\ c\le -M_2. 
\end{eqnarray*}
or the other way round.
Within the Region $D$, when $s^2>t$, we will
additionally consider subregions
$\Sigma$ with respect to  the variables $s$ and $t$.  However, the constants in the
estimate $
 h_{{\Sigma},c}\asymp E$ on $ \Sigma $ 
 will not depend on $\Sigma$ nor on $c\in\R$.

In order to apply the parabolic minimum principle on the upper Region $U$, we will construct a function $h_{U,c}$
 from the auxiliary functions
 by "gluing" them with a smooth partition of unity
 $(w_{\mathcal S})$
$$
h_{U,c}=\sum_{\mathcal S} w_{\mathcal S}h_{{\mathcal S},c}
 $$
 In practice we will always consider the constant $c>0$
 and multiply it by the sign $\mp$ denoted $\S$.
 We will skip $c$ from the notation 
 $h_{{\mathcal S},c}$ and write 
 $h_{\mathcal S}$.
 
Some more involved computations of the derivatives were performed with Maple.

\subsection{Subdivision into regions}
~\\

In what follows, $K_0$ corresponds to the constant in Lemma \ref{G}.
We will consider five regions (see Figure \ref{overlap}):

\begin{itemize}
\item[Region 0:\,] $0<t\leq M$  for a fixed $M>0$,

\item[Region A:] $0<t\leq 2\,r$,

\item[Region B:] $t \geq T= (R^2+1)^{3/2}$, $t\geq r\geq s$ and  $r\leq R$ for a fixed $R>0$,

\item[Region C:] $t \geq 1$, $t\geq r\geq s$ and  $t\geq r\,s/2$, $r\geq R_0$ where $R_0$ is given later in \eqref{def:R0},

\item[Region D:]  $t \geq 1$, $s\leq r\leq t\leq rs$, $r\geq R_0$.
\end{itemize}

\setlength{\unitlength}{0.5cm}

\begin{figure}[htbp]
\begin{picture}(10,10)(1,0)
\put(4.5,1.3){\fbox{\tiny 0}}
\put(6.5,4.5){\fbox{\tiny A}}
\put(1.1,7){\fbox{\tiny B}}
\put(3.3,9.0){\fbox{\tiny C}}
\put(5.0,7.0){\fbox{\tiny D}}

\put(0,0){\vector(1,0){10}}
\put(0,0){\vector(0,1){10}}

\put(0,3){\line(1,0){10}}

\put(3,3){\line(0,1){7}}
\put(3,3){\line(1,1){7}}
\put(3,6){\line(1,2){2}}

\put(-0.5,9.5){$\scriptstyle t$}
\put(9.7,-0.5){$\scriptstyle r$}
\put(10.0,9.5){$\scriptstyle t=r$}
\put(5.0,9.4){$\scriptstyle t=2\,r$}
\put(-0.7,2.8){$\scriptstyle M$}
\put(2.7,-0.7){$\scriptstyle R$}
\put(3.0,-0.2){\line(0,1){0.4}}
\end{picture}
\caption{Sketch of the regions without overlap\label{overlap}}
\end{figure}
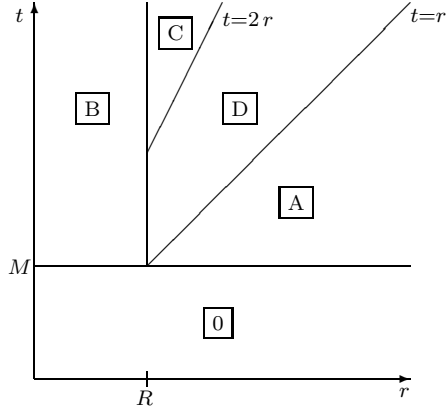

By choosing appropriate parameters, it can be seen that these regions cover all cases.\\

\bigskip

\subsection{Region 0:  $0<t\leq M$ for a fixed $M>0$}
~\\

Define $\delta_{\Op}(r)=\sinh^{2\,k}r$ and $\delta_{\Du}(r)=r^{2\,k}$.  In Region 0, the estimate {\eqref{estim:conjecture}
from Theorem \ref{Main}} is equivalent to 
\begin{align*}
h_{\Op}(r,s,t)&\asymp t^{-\frac{d}{2}} \,e^{\frac{-|r-s|^2}{4t}}\,e^{-k\,(r+s)} \,(1+r)\,(1+s)\,\frac{(1+r)^{k-1}}{(t+r\,s)^{k}}\\
&\asymp e^{-k^2\,t} \, e^{-k\,s} \,(1+s)\,\frac{\delta_{\Op}^{-1/2}(r)}{\delta_{\Du}^{-1/2}(r)} \,h_{\Du}(r,s,t)
\ {\eqdef}\
h_0(r,s,t)
\end{align*}
where  $h_{\Du}(r,s,t)$ is the Dunkl heat kernel in three variables.

The {modified} radial Laplacians for the 
Opdam-Cherednik
and Dunkl {settings} are respectively
\begin{align*}
 L_{\Op}\,=\frac{\partial^2~}{\partial r^2}+2\,k\,\coth r\,\frac{\partial~}{\partial r}=\delta^{-1/2}\circ\frac{\partial^2~}{\partial r^2}\circ \delta^{1/2}-w_{\Op}\\
L_{\Du}=\frac{\partial^2~}{\partial r^2}+2\,k\,\frac{1}{r}\,\frac{\partial~}{\partial r}=\delta_{\Du}^{-1/2}\circ\frac{\partial^2~}{\partial r^2}\circ \delta_{\Du}^{1/2}-w_{\Du}\\
\end{align*}
where $w_{\Op}=k^2+\frac{k^2-k}{\sinh^2r}$ and $w_{\Du}=\frac{k^{2}-k}{r^{2}}$. Now,
\begin{align*}
\frac{D_{\Op}\,h_0}{h_0}&=\frac{D_{\Du}\,(e^{-k^2\,t} \,h_{\Du})}{e^{-k^2\,t} \,h_{\Du}}+w_{\Op}-w_{\Du}
=-(k^2-k)\,\left(\frac{1}{r^2}-\frac{1}{\sinh^2r}\right)+k^2-k^2\\
&=O\left(\frac{1}{(1+r)^2}\right)
\end{align*}
for all $r\geq 0$. For simplicity, we will write 
\begin{align}
\left|{\frac{D_{\Op}\,h_0}{h_0}}\right|\leq {H\over (1+r)^2}\label{boundA}
\end{align}
for some fixed $H>0$ independent of $r$, $s$, $t$.

We apply the parabolic minimum principle for $ (r,t)\in [0,\infty)\times [0,M]$ to the expression
\begin{align*}
u(r,s,t) {\ \eqdef\ } e^{c\,t/(1+t)}\,h_0(r,s,t)-h_{\Op}(r,s,t).
\end{align*}

Since
\begin{align*}
\frac{D_{\Op}u}{e^{c\,t/(1+t)}\,h_0}=\frac{c}{(1+t)^2}+\frac{D_{\Op}h_0}{h_0}\geq \frac{c}{(1+M)^2}-{H\over (1+r)^2}\geq0 \hbox{~provided $c\geq H\,(1+M)^2$},
\end{align*}
and
\begin{align*}
\left\lbrace
\begin{array}{cl}
u(r,s,0)&=\delta(r,s)-\delta(r,s)=0,\\ 
u(r,s,t)&=0\qquad\qquad\hbox{as $r\to\infty$},
\end{array}
\right.
\end{align*}
by the minimum principle, we have $u(r,s,t)\geq 0$ \emph{i.e.} $h_{\Op}(r,s,t)\leq e^{c\,t/(1+t)}\,h_0(r,s,t)$ on $ [0,\infty)\times [0,\infty)\times [0,M]$.  Thus we deduce that
\begin{align*}
h_{\Op}(r,s,t)&\leq e^{c}\,h_0(r,s,t)\\
\end{align*}
on $[0,\infty)\times[0,\infty)\times [0,{M}]$.

By interchanging the Dunkl and the 
Opdam-Cherednik roles, we obtain the converse estimate
\begin{align*}
h_{\Op}(r,s,t)\geq e^{c} \,h_0(r,s,t)
\end{align*}

\bigskip

\subsection{Upper regions A, B, C, D}

\noindent\paragraph{\bf Region A:  $0<t\leq 2\,r$}~\\

We define $h_A$ with the same expression as for $h_0$.  Hence, using \eqref{boundA}, we have
\begin{align*}
\left|\frac{D_{\Op}\,h_A}{h_A}\right|\leq H/(1+r)^2\leq H/(1+t/2)^2\leq 4\,H/t^2.
\end{align*}

So if we replace $h_A$ by $e^{-4\,\S\,H/t}\,h_A$, $D_{\Op}\,h_A$ will have the same sign as $\S$.

\bigskip

\noindent\paragraph{\bf Region B: $t \geq T= (R^2+1)^{3/2}$, $t\geq r\geq s$ and  $r\leq R$ for a fixed $R>0$}~\\

In that region, the estimate from Conjecture \ref{conjecture} is equivalent to 
\begin{align*}
t^{-3/2}\,e^{-k^{2} t} \phi_0 (r)\,
\asymp h_B^+
\eqdef
t^{-3/2}\,e^{-k^{2} t} \phi_0 (r)\,\,\overbrace{(1+\left(R^{2}+1\right)^{\frac{3}{2}}-\left(r^{2}+1\right)^{\frac{3}{2}})}^{Q(r)}
\end{align*}
for $0\leq r\leq R$.

Now, since $G(r)\geq 0$ by Lemma \ref{G} and $t\geq T$,
\begin{align*}
\frac{D_{\Op}\,h_B^+}{h_B^+}&=
{\frac{6 \sqrt{r^{2}+1}}{1+\left(R^{2}+1\right)^{\frac{3}{2}}-\left(r^{2}+1\right)^{\frac{3}{2}}}\,G(r)
+\overbrace{\frac{3 \left(2 r^{2}+1\right)}{\sqrt{r^{2}+1}\, \left(1+\left(R^{2}+1\right)^{\frac{3}{2}}-\left(r^{2}+1\right)^{\frac{3}{2}}\right)}}^{\hbox {increasing in $r$}}-\frac{3}{2 t}}\\
&\geq \frac{3}{\left(R^{2}+1\right)^{\frac{3}{2}}}-\frac{3}{2\, (R^2+1)^{3/2}}\geq 0.
\end{align*}

On the other hand, if
\begin{align*}
h_B^-
\eqdef
t^{-3/2}\,e^{-k^{2} t} \phi_0(r)
\end{align*}
then 
\begin{align*}
\frac{D_{\Op}\,h_B^-}{h_B^-}=
{-\frac{3}{2\,t}}\leq 0.
\end{align*}

From now on, we will assume \begin{equation}\label{def:R0}
R_0=\max\{1,4\,K_0-1\}.
\end{equation}

\bigskip

\noindent\paragraph{\bf Region C: $t \geq 1$, $t\geq r\geq s$ and  $t\geq r\,s/2$, $r\geq R_0$}~\\

In that region, the estimate from Conjecture \ref{conjecture} is equivalent to 
\begin{align*}
h_C^{\pm}
\eqdef
\frac{e^{-\frac{d r s}{t}} e^{\frac{\S  c \sqrt{r^{2}+1}}{t}} t^{-\frac{3}{2}+k} \phi_0  (r)\, \phi_0  \left(s \right) e^{-\frac{\left(r -s \right)^{2}}{4 t}} e^{-k^{2} t}}{t^{k}}.
\end{align*}

If we apply the operator $D_{\Op}$ and use Lemma \ref{G}, we obtain
\begin{align}
\frac{D_{\Op}\,h_C^{\pm}}{h_C^{\pm}}&=
{\left(-\frac{2 \S  c}{t \sqrt{r^{2}+1}}+\frac{1}{t}-\frac{s}{t r}+\frac{2 d s}{t r}\right) G +\frac{2 d s \S  c r}{t^{2} \sqrt{r^{2}+1}}
+\frac{\S  c \,r^{2}}{t \left(r^{2}+1\right)^{\frac{3}{2}}}-\frac{\S ^{2} c^{2} r^{2}}{t^{2} \left(r^{2}+1\right)}+\frac{d \,s^{2}}{t^{2}}}
\nonumber\\&
{-\frac{\S  c \sqrt{r^{2}+1}}{t^{2}}-\frac{\S  c}{t \sqrt{r^{2}+1}}-\frac{d^{2} s^{2}}{t^{2}}-\frac{\S  c r s}{t^{2} \sqrt{r^{2}+1}}-\frac{1}{t}+\frac{\S  c \,r^{2}}{t^{2} \sqrt{r^{2}+1}}} \nonumber\\
&=
{-\frac{\S ^{2} r^{2}}{t^{2} \left(r^{2}+1\right)} \,c^2
+\S \,\left(\frac{2 dr s}{\sqrt{r^{2}+1}\, t^{2}}-\frac{r^{3} s +2 r^{2} t +r^{2}+r s +3 t +1}{t^{2} \left(r^{2}+1\right)^{\frac{3}{2}}}\right)\,c\label{dgen}
-\frac{2 \S }{\left(1+r \right) \sqrt{r^{2}+1}\, t}\,K\,c}\\&
{+\frac{2 d s +r -s}{t r \left(1+r \right)}\,K
-\frac{d^{2} s^{2}}{t^{2}}+\frac{s \left(r s +2 t \right) d}{r \,t^{2}}-\frac{s}{t r}.}\nonumber
\end{align}

If we set $d=1/2$ and $\S =-1$ and using $r\geq 4\,K_0-1$ then \eqref{dgen} becomes, 
\begin{align*}
\frac{D_{\Op}\,h_C^+}{h_C^+}
&=
-\frac{ r^2}{t^{2} (r^2+1)} \,c^2
+\frac{2 r^{2} t +r^{2}+3 t +1}{t^{2} \left(r^{2}+1\right)^{\frac{3}{2}}}\,c
+\frac{2 }{\left(1+r \right) \sqrt{r^{2}+1}\, t}\,K\,c
+{\frac{1}{t  \left(1+r \right)}\,K}
+\frac{s^{2}}{4 t^{2}}
\\&\geq
-\frac{ 1}{t^{2} } \,c^2
+\frac{2}{t \left(r^{2}+1\right)^{\frac{1}{2}}}\,c
-\frac{2 }{\left(1+r \right) \sqrt{r^{2}+1}\, t}\,K_0\,c
-{\frac{1}{t  \left(1+r \right)}\,K_0}
+\frac{s^{2}}{4 t^{2}}\\
&\geq -\frac{ 1}{t^{2} } \,c^2
+\frac{2}{t \left(r^{2}+1\right)^{\frac{1}{2}}}\,c
-\frac{1 }{2 \sqrt{r^{2}+1}\, t}\,c
-\frac{1}{t  \sqrt{r^{2}+1}}\,K_0\\
&=-\frac{ 1}{t^{2} } \,c^2
+\frac{c-K_0}{t \left(r^{2}+1\right)^{\frac{1}{2}}}.
\end{align*}

If we choose $c\geq K_0+1$, we have $D_{\Op}\,(e^{-c^2/t}\,h_C^+)\geq 0$.

\bigskip

If we set $d=0$ and $\S =1$ then using $r\geq 4\, K_0-1$, \eqref{dgen} becomes, 
\begin{align*}
\frac{D_{\Op}\,h_C^-}{h_C^-}
&=
-\frac{ r^2}{t^{2} (r^2+1)} \,c^2
-\frac{r^{3} s +2 r^{2} t +r^{2}+r s +3 t +1}{t^{2} \left(r^{2}+1\right)^{\frac{3}{2}}}\,c
-\frac{2 }{\left(1+r \right) \sqrt{r^{2}+1}\, t}\,K\,c
+\frac{r -s}{t r \left(1+r \right)}\,K
-\frac{s}{t r}\\
&\leq
-\frac{r^{3} s +2 r^{2} t +r^{2}+r s +3 t +1}{t^{2} \left(r^{2}+1\right)^{\frac{3}{2}}}\,c
+\frac{2}{t \sqrt{r^{2}+1}\, \left(1+r \right)}\,K_0\,c
+\frac{r }{t r \left(1+r \right)}\,K_0\\
&\leq
-\frac{1}{t\,\sqrt{r^2+1}}\,c
+\frac{1}{2\,t \sqrt{r^{2}+1}}\,c
+\frac{1}{t\,\sqrt{r^2+1}}\,K_0\leq
-\frac{c-1/2-K_0}{t\,\sqrt{r^2+1}}\leq0
\end{align*}
provided $c\geq 1/2+K_0$.

\bigskip
 
\noindent\paragraph{\bf Region D: $t \geq 1$, $s\leq r\leq t\leq rs$, $r\geq R_0$}~\\

Note that this implies that $s\geq 1$.  Indeed, $s=r s/r\geq t/r\geq 1$.  We also have $t\leq rs\leq r^2\leq r^3$.

\bigskip

\noindent\paragraph{\bf Subregion $D_1 = D$ with $t\geq s^2/2$}~\\

In that region, the estimate from Conjecture \ref{conjecture} is equivalent to 
\begin{align*}
h_{D_1}^\pm\eqdef \frac{e^{\frac{\S  c \,s^{2}}{t}} t^{-\frac{3}{2}+k} \phi_0(r)\, \phi_0  \left(s \right) e^{-\frac{\left(r -s \right)^{2}}{4 t}} e^{-k^{2} t}}{\left(r s \right)^{k}}.
\end{align*}

If we apply the operator $D_{\Op}$ and use Lemma \ref{G}, we obtain
\begin{align}
\frac{D_{\Op}\,h_{D_1}^\pm}{h_{D_1}^\pm}&=
{\left(\frac{1}{t}-\frac{s}{t r}+\frac{2 k}{r^{2}}\right) G 
+\frac{-\S  \,c \,s^{2}}{t^{2}}-\frac{1}{t}+\frac{k s}{t r}-\frac{k^{2}}{r^{2}}-\frac{k}{r^{2}}}\nonumber\\
&=
{\frac{\left(2 k t +r^{2}-r s \right) K}{t \,r^{2} \left(1+r \right)}+\frac{\S  \,c \,s^{2} r^{2}-k^{2} t^{2}+k r s t +k \,t^{2}-r s t}{t^{2} r^{2}}.}\label{dgen2}
\end{align}

\bigskip

If we set $\S =-1$ then \eqref{dgen2} becomes
\begin{align*}
\frac{D_{\Op}\,h_{D_1}^+}{h_{D_1}^+}
&=\frac{\left(2 k t +r^{2}-r s \right) K}{t \,r^{2} \left(1+r \right)}+\frac{c \,s^{2} r^{2}-k^{2} t^{2}+k r s t +k \,t^{2}-r s t}{t^{2} r^{2}}\\
&\geq -\frac{\left(2 k r^2 +r^{2}+r^2 \right) K_0}{t \,r^{3}}+\frac{c \,s^{2} r^{2}-k^{2} t^{2}-t rs}{t^{2} r^{2}}
\geq-\frac{\left(2 k +2 \right) K_0}{t \,r}+\frac{(c-k^2-1) \,r^{2} s^2}{t^{2} r^2}\\
&=\frac{(c-k^2-1)\,r s^2-\left(2 k +2 \right) K_0\,t}{r\,t^2}
\geq\frac{(c-k^2)-\left(2 k +2 \right) K_0}{r\,t}\,r s^2\geq 0
\end{align*}
($t \leq rs^2=rs s$ since $s\geq 1$) as long as $c\geq k^2+(2k+2)\,K_0$.

\bigskip

If we set $\S =1$ then \eqref{dgen2} becomes
\begin{align*}
\frac{D_{\Op}\,h_{D_1^-}}{h_{D_1}^-}&=\frac{\left(2 k t +r^{2}-r s \right) K}{t \,r^{2} \left(1+r \right)}+\frac{-c \,s^{2} r^{2}-k^{2} t^{2}+k r s t +k \,t^{2}-r s t}{t^{2} r^{2}}\\
&\leq \frac{\left(2 k t +r^{2}+r^2 \right) K_0}{t \,r^{3}}+\frac{-c \,s^{2} r^{2}+k r s t +k \,t^{2}}{t^{2} r^{2}}
\leq \frac{\left(2 k  +1+1 \right) K_0}{t \,r}+\frac{-c \,s^{2} r^{2}+k r s t +k \,t^{2}}{t^{2} r^{2}}\\
&=\frac{-c \,s^{2} r^{2}+k r s t +k \,t^{2}+\left(2 k  +2\right) K_0\,t\,r}{t^{2} r^{2}}
\leq \frac{-c \,s^{2} r^{2}+k (rs)^2 +k \,(rs)^2+\left(2 k  +2\right) K_0\,r^2 s^2}{t^{2} r^{2}}\\
&=  \frac{-c +2k+\left(2 k  +2\right) K_0}{t^{2}}\,s^2\leq 0
\end{align*}
as long as $c\geq 2k+(2k+2)\,K_0$.

\bigskip

\paragraph{\bf Subregion $D_2 = D$ with $t\leq s^2$}~\\

Recall from Section \ref{sec:cover} that $T_a=\{t\colon s^a/4\leq t\leq 2\,s^a\}$.
In the subregion $D_2 \cap \{(r,s,t)\,\colon\, t\in T_a  \}$,  the estimate from Conjecture \ref{conjecture} is equivalent to 
\begin{align}
h_{D_2,a}^\pm
\eqdef\ 
\frac{e^{\frac{-\S  c t}{s^{a}}} t^{-\frac{3}{2}+k} \phi_0(r) \phi_0(s) e^{-\frac{\left(r -s \right)^{2}}{4 t}} e^{-k^{2} t}}{\left(r s \right)^{k}}\label{Sa}
\end{align}

If we set $\S =-1$ then $1\leq e^{ct/s^a}\leq e^{2\,c}$ and \eqref{Sa} becomes
\begin{align*}
h_{D_2,a}^+=\frac{e^{-\frac{ c t}{s^{a}}} t^{-\frac{3}{2}+k} \phi_0(r) \phi_0(s) e^{-\frac{\left(r -s \right)^{2}}{4 t}} e^{-k^{2} t}}{\left(r s \right)^{k}}
\end{align*}
and
\begin{align*}
\frac{D_{\Op}\,h_{D_2,a}^+}{h_{D_2,a}^+}&=\left(\frac{1}{t}-\frac{s}{t r}+\frac{2 k}{r^{2}}\right) G +\frac{ c}{s^{a}}-\frac{1}{t}+\frac{k s}{t r}-\frac{k^{2}}{r^{2}}-\frac{k}{r^{2}}\\
&=\frac{\left(2 k t +r^{2}-r s \right) K}{t \,r^{2} \left(1+r \right)}+\frac{-k^{2} s^{a} t +k s \,s^{a} r +c \,r^{2} t +k \,s^{a} t -s r \,s^{a}}{t \,r^{2} s^{a}}\\
&\geq-\frac{\left(2 k r^2 +2\,r^{2} \right) K_0}{t \,r^3}+\frac{-k^{2} s^{a} r^2 +c \,r^{2} s^a/4  -r^2 \,s^{a}}{t \,r^{2} s^{a}}\\
&=-\frac{\left(2 k+2\right) K_0}{tr}+\frac{-k^{2}  +c /4  -1 }{t }=\frac{(c /4 -k^{2} -1)\,r- \left(2 k+2\right) K_0}{r\,t}\\
&\geq \frac{c /4 -k^{2} -1- \left(2 k+2\right) K_0}{r\,t}\geq0
\end{align*}
as long as $c\geq 4\,(k^2+1+\left(2 k+2\right) K_0)$.

If we set $\S =1$ then $e^{-c/2}\leq e^{-ct/s^a}\leq 1$ and \eqref{Sa} becomes
\begin{align*}
\frac{D_{\Op}\,h_{D_2,a}^-}{h_{D_2,a}^-}&=\left(\frac{1}{t}-\frac{s}{t r}+\frac{2 k}{r^{2}}\right) G +\frac{- c}{s^{a}}-\frac{1}{t}+\frac{k s}{t r}-\frac{k^{2}}{r^{2}}-\frac{k}{r^{2}}\\
&=\frac{\left(2 k t +r^{2}-r s \right) K}{t \,r^{2} \left(1+r \right)}-\frac{k^{2} s^{a} t -k s \,s^{a} r +c \,r^{2} t -k \,s^{a} t +s r \,s^{a}}{t \,r^{2} s^{a}}\\
&\leq \frac{\left(2 k r^2 +r^{2}+r^2 \right) K_0}{t \,r^{2} \left(1+r \right)}-\frac{ -k s \,s^{a} r +c \,r^{2} s^a/4 -k \,s^{a} t }{t \,r^{2} s^{a}}
=\frac{\left(2 k  +2 \right) K_0}{t \,r}-\frac{ -k s r +c \,r^{2}/4 -kt }{t \,r^{2} }\\
&\leq\frac{\left(2 k  +2 \right) K_0}{t \,r}-\frac{ (c/4-2k)\,r^2}{t \,r^{2} }
=\frac{ (-c/4+2k)\,r+\left(2 k  +2 \right) K_0}{t\,r}\leq \frac{ -c/4+2k+\left(2 k  +2 \right) K_0}{t\,r}\leq0
\end{align*}
as long as $c\geq 4\,(2k +\left(2 k  +2 \right) K_0)$.

\subsection{Gluing of regions A, B, C, D}

\begin{note}\label{chi}
In what follows, $\chi$ is a smooth bump (cut-off) non-negative function on $\R$ such that $\chi(t)=\left\lbrace\begin{array}{cl}1&t\leq 1\\0&t\geq 2\end{array}\right.$.
Since $\chi $ is smooth and constant outside the interval $[1,2]$, for each $j\geq 0$, there exists a constant $M_j>0$ such that the $j$-th derivative $|\chi^{(j)}(t)|\leq M_j$ for all $t$.
\end{note}

\begin{remark}\label{glue}
Given the functions $h_1$, \dots, $h_p$, we will glue them using a partition of unity $w_1(r,s,t)$, \dots, $w_p(r,s,t)$  consisting of smooth non-negative functions such that $\sum_{j=1}^p\,w_j(r,s,t)=1$.  We will ensure that the following properties hold true:
\begin{itemize}
\item[(i)] the functions $w_j(r,s,t)$ and their first and second order derivatives are universally bounded,
\item[(ii)] whenever $w_j(r,s,t)>0$, the function $h_j$ is equivalent to \eqref{estim:conjecture} and $h_j=O(1/t^2)$ for this value of $(r,s,t)$,

\item[(iii)] whenever $w_j(r,s,t)>0$, then $D_{\Op}\,h_j$ is of the correct sign (``non-negative'' or ``non-positive'' given the context).
\end{itemize}
\end{remark}

\bigskip

\noindent\paragraph{\bf Gluing of the subregion $D_2 = D$ with\ $t\leq s^2$}
~\\

Let $m\in\N\setminus\{0\}$ and $u=1$, \dots, $2^{m+1}-2^m+1$. Let us call $h_{u,m}^\pm$ the function 
$h_{D_2,a}$ from \eqref{Sa}, with $a=(2^m+u-1)/2^m$, constructed so that $D_{\Op}\,h_{u,m}^\pm$ is of the correct sign (``non-negative'' or ``non-positive'' given the context) on the interval
$T_{(2^m+u-1)/2^m}$.   Assume $s\in [4^{2^{m-1}},4^{2^{m}}]$ and consider 
$t\in S_{2^m/2^m}\cup S_{(2^m+1)/2^m}$. 
Write $h^{\pm}=u_1(t)\,h_{1,m}^{\pm}+(1-u_1(t))\,h_{2,m}^{\pm}$ where 
\begin{align*}
u_1(t)=\chi\left(\frac{t+2\,s-s^{(2^m+1)/2^m}/2}{2\,s-s^{(2^m+1)/2^m}/4}\right).
\end{align*}
(refer to Notation \ref{chi}).  We have $|u_1'(t)|\leq M_1/(8\,s-s^{(2^m+1)/2^m}/4)=M_1/\left((s/4)\,(8-s^{1/2^m}/4)\right)\leq M_1/((1/4)\,(8-4))=M_1$ on $[4^{2^{m-1}},4^{2^{m}}]$.
In the same manner,  $|u_1''(t)|\leq  M_2$ on $[4^{2^{m-1}},4^{2^{m}}]$.

\begin{itemize}
\item If $s\leq t\leq s^{(2^m+1)/2^m}/4$ then $t\in S_{2^m/2^m}\subset T_{2^m/2^m}$, $u_1(t)=1$ and $h=h_{1,m}$.
\item If $2\,s\leq t\leq 2\,s^{(2^m+1)/2^m}$ then  $t\in S_{(2^m+1)/2^m}\subset T_{(2^m+1)/2^m}$, $u_1(t)=0$ and $h=h_{2,m}^\pm$.  
\item If $\max\{s,s^{(2^m+1)/2^m}/4\}\leq t\leq 2\,s$ then  $t\in S_{2^m/2^m}\cap T_{(2^m+1)/2^m}$ and $h$ is a convex combination of $h_{1,m}^\pm$ and $h_{2,m}^{\pm}$.
\end{itemize}

We continue the construction inductively. Suppose that $h^\pm$ has been constructed in this manner for $t\in \cup_{i=2^m}^{N}\,S_{i/2^m}\ \eqdef\ S(N)$ with $N<2^{m+1}$
and that $u_i(t)>0$ implies $t\in S_{(2^m+i)/2^m}\cup T_{(2^m+i+1)/2^m}$. From Lemma \ref{cover}, this implies that no more than three $u_i$'s can be nonzero at the same time.  We then define $\tilde{h}^{\pm}=v_1(t)\,h^\pm+(1-v_1(t))\,h_{N+1,m}^\pm$
where 
\begin{align*}
v_1(t)=\chi\left(\frac{t+2\,s^{M/2^m}-s^{(M+1)/2^m}/2}{2\,s^{M/2^m}-s^{(M+1)/2^m}/4}\right).
\end{align*}
\begin{itemize}
\item If $s\leq t\leq s^{(N+1)/2^m}/4$ then $t\in S(N)$,
$v_1(t)=1$ and $\tilde{h}^{\pm}=h^{\pm}$.
\item If $s^{(N+1)/2^m}/4\leq t\leq 2\,s^{(N+1)/2^m}$ then  $t\in S_{(2^m+1)/2^m}\subset T_{(2^m+1)/2^m}$, $u_1(t)=0$ and $\tilde{h}^{\pm}= h_{N+1,m}^\pm $.  
\item If $\max\{s, s^{(N+1)/2^m}/4\}\leq t\leq 2\,s^{(N+1)/2^m}$ then  $t\in S_{2^m/2^m}\cap T_{(2^m+1)/2^m}$ and 
$\tilde h^\pm$ is a linear combination of $h^\pm$ and $h_{N+1,m}^\pm $.
\end{itemize}
We have $|v_1'(t)|\leq M_1/(2\,s^{M/2^m}-s^{(M+1)/2^m}/4)=M_1/((s^{M/2^m}/4)\,(8-s^{1/2^m}))\leq M_1/((1/4)(8-4))=M_1$ on $[4^{2^{m-1}},4^{2^{m}}]$.
In the same manner,  $|v_1''(t)|\leq  M_2$ on $[4^{2^{m-1}},4^{2^{m}}]$.

Continuing the process, we have a partition of unity $w_1$, \dots, $w_{2^{m+1}-2^m+1}$
so that (see Covering Lemma)
no more than three of the $w_i$'s (and therefore no more than three $w_i'$'s and $w_i''$'s) are nonzero for a given value of $t$:
\begin{align*}
h^\pm=\sum_{i=1}^{2^{m+1}-2^m+1}\,w_i(t)\,h_{2^m+i-1,m}^\pm.
\end{align*}

Assume that $Dh_{i,m}^+\geq 0$ for each $i$.  We have
\begin{align*}
D_{\Op}h^+&=\sum_{i=1}^{2^{m+1}-2^m+1}\,w_i(t)\,D_{\Op}h_{i,m}^++\sum_{i=1}^{2^{m+1}-2^m+1}\,w_i'(t)\,h_{i,m}^+
\geq \sum_{i=1}^{2^{m+1}-2^m+1}\,w_i'(t)\,h_{i,m}^+\\
&\geq -3\,M_1\,\max_{{i=1}\dots{2^{m+1}-2^m+}}h_{i,m}^+
=O(1/t^2).
\end{align*}

If $c>0$ is large enough, $D_{\Op}(e^{-c/t}\,h^+)\geq 0$.  If $D_{\Op}h_{i,m}^-\leq 0$ for each $i$, we proceed in a similar manner except that we multiply $h^-$ by  $e^{c/t}$.
Note that the function $h_{D_2}^\pm
\ \eqdef\ h^\pm$ thus built satisfies
\begin{align*}
 h_{D_2}^\pm=  C(s,t)\,
\frac{ t^{-\frac{3}{2}+k} \phi_0(r) \phi_0(s) e^{-\frac{\left(r -s \right)^{2}}{4 t}} e^{-k^{2} t}}{\left(r s \right)^{k}}.
\end{align*}
where $C(s,t)$ is bounded and does not depend on $r$.\\

\bigskip

\noindent\paragraph{\bf Gluing of the Region D}~\\

We now glue the function $h_{D_1}^\pm$ constructed
for the subregion $D_1$ of the region D such that $t\geq s^2/2$, with the function $h_{D_2}^\pm$ (the one we just constructed) for the subregion $t\leq s^2$.
Let 
\begin{align*}
h_D^\pm\ \eqdef\ \chi(s^2/t)\,h_{D_1}^\pm+(1-\chi(s^2/t))\,h_{D_2}^\pm.
\end{align*}
Note that
\begin{align*}
D_{\Op} h_D^\pm&=\chi(s^2/t)\,D_{\Op}h_{D_1}^\pm+(1-\chi(s^2/t))\,D_{\Op}h_{D_2}^\pm-\chi'(s^2/t)\,(s^2/t^2)\,h_{D_1}^\pm+\chi'(s^2/t)\,(s^2/t^2)\,h_{D_2}^\pm\\
\text{and}\\
&|-\chi'(s^2/t)\,(s^2/t^2)\,h_{D_1}^\pm+\chi'(s^2/t)\,(s^2/t^2)\,h_{D_2}^\pm|
\leq M_1\,(h_{D_1}^\pm+h_{D_2}^\pm)
\end{align*}
which is $O(1/t^2)$ ($\chi'(s^2/t)=0$ unless $1\leq s^2/t\leq 2$).  We proceed as before and rename the resulting function $h_D^\pm$.\\

\bigskip

\noindent\paragraph{\bf Gluing of the Region C and Region D}~\\

We now connect the regions $C$ and $D$.  Write
\begin{align*}
h_{CD}^\pm \ \eqdef\ \chi(r s/t)\,h_C^\pm+(1-\chi(r s/t))\,h_D^\pm\\
&+\end{align*}
and note that $D_{\Op}h_{CD}^\pm=\chi(r s/t)\,D_{\Op}h_C^\pm+(1-\chi(r s/t))\,D_{\Op}h_D^\pm+R(r,s,t)$ where
\begin{align*}
R(r,s,t)&
{=\left(-\frac{\chi''(r s/t) s^{2}}{t^{2}}-\frac{\chi'(r s/t) r s}{t^{2}}-\frac{2 k \cosh  \left(r \right) \chi'(r s/t) s}{\sinh  \left(r \right) t}\right) h_C^\pm (r,s,t)}\\
&
{+\left(\frac{\chi''(r s/t) s^{2}}{t^{2}}+\frac{\chi'(r s/t) r s}{t^{2}}+\frac{2 k \cosh  \left(r \right) \chi'(r s/t) s}{\sinh  \left(r \right) t}\right) h_D^\pm (r,s,t)}\\
&
{-\frac{2 \chi'(r s/t) s \left(\frac{\partial}{\partial r}h_C^\pm (r,s,t)\right)}{t}+\frac{2 \chi'(r s/t) s \left(\frac{\partial}{\partial r}h_D^\pm (r,s,t)\right)}{t}.}
\end{align*}
Note that $\chi'(r s/t)$ and $\chi''(r s/t)$ are zero unless $1<rs/t<2$ and that they are universally bounded by $M_1$ and $M_2$ respectively. Furthermore, $s\leq r\leq t$ and $r\geq1$. To show that $R(r,s,t)=O(1/t^2)$, it suffices then to examine $\frac{\partial}{\partial r}h_C^\pm (r,s,t)$ and $\frac{\partial}{\partial r}h_D^\pm (r,s,t)$.  It is not difficult to see using Lemma \ref{G} and Lemma \ref{crude} that both these functions are $O(1/t^2)$. 

\bigskip

\noindent\paragraph{\bf Gluing of Region B with Regions C and D}~\\

We choose $R>R_0$ (in the definition of Region $B$) say $R=R_0+1$.
Let
\begin{align*}
h_{BCD}^\pm
\ \eqdef\ 
\chi (r -R_{0}+1) h_{B}^\pm(r,s,t)+(1-\chi (r -R_{0}+1) ) h_{CD}^\pm(r,s,t).
\end{align*}

We have $D_{\Op}\,h_{BCD}^\pm=h_{BCD}^\pm=\chi (r -R_{0}+1) D_{\Op}\,h_{B}^\pm(r,s,t)+(1-\chi (r -R_{0}+1) ) D_{\Op}\,h_{CD}^\pm(r,s,t)+R(r,s,t)$ where 
\begin{align*}
R(r,s,t)&
{=-\chi''(r-R_0+1) h_{B}^\pm(r,s,t)-2 \chi'(r-R_0+1) \left(\frac{\partial}{\partial r}h_{B}^\pm(r,s,t)\right)}\\
&
{+\chi''(r-R_0+1) h_{CD}^\pm(r,s,t)+2 \chi'(r-R_0+1) \left(\frac{\partial}{\partial r}h_{CD}^\pm(r,s,t)\right)}\\
&
{-\frac{2 k \cosh \! \left(r \right) \chi'(r-R_0+1) h_{B}^\pm(r,s,t)}{\sinh \! \left(r \right)}+\frac{2 k \cosh \! \left(r \right) \chi'(r-R_0+1) h_{CD}^\pm(r,s,t)}{\sinh \! \left(r \right)}.}
\end{align*}

Note that $\chi'(r -R_{0}+1)$ and $\chi''(r -R_{0}+1)$ are zero unless $1<r -R_{0}+1<2$ and that they are universally bounded by $M_1$ and $M_2$ respectively.  To show that $R(r,s,t)=O(1/t^2)$, it suffices then to examine $\frac{\partial}{\partial r}h_B^\pm (r,s,t)$ and $\frac{\partial}{\partial r}h_{CD}^\pm (r,s,t)$.  It is not difficult to see using Lemma \ref{G} and Lemma \ref{crude} that both these functions are $O(1/t^2)$.  
\\

\bigskip

\noindent\paragraph{\bf Gluing of Regions A, B, C, D}~\\

\bigskip

We take 
\begin{align*}
h_{ABCD}^\pm&
\ \eqdef\ \chi (t/r) h_{A}^\pm(r,s,t)+(1-\chi (t/r) ) h_{BCD}^\pm(r,s,t).
\end{align*}

We have $h_{ABCD}^\pm=\chi (t/r) D_{\Op}\,h_{A}^\pm(r,s,t)+(1-\chi (t/r) ) D_{\Op}\,h_{BCD}^\pm(r,s,t)+R(r,s,t)$ where
\begin{align*}
R(r,s,t)&=
{ \left(\frac{\chi'\! \left(\frac{t}{r}\right)}{r}-\frac{\chi''\! \left(\chi \right)\! \left(\frac{t}{r}\right) t^{2}}{r^{4}}-\frac{2 \chi'\! \left(\frac{t}{r}\right) t}{r^{3}}+\frac{2 k \cosh \! \left(r \right) \chi'\! \left(\frac{t}{r}\right) t}{\sinh \! \left(r \right) r^{2}}\right) h_{A}\! \left(r , s , t\right)}\\
&
{+\left(-\frac{\chi'\! \left(\frac{t}{r}\right)}{r}+\frac{\chi''\! \left(\chi \right)\! \left(\frac{t}{r}\right) t^{2}}{r^{4}}+\frac{2 \chi'\! \left(\frac{t}{r}\right) t}{r^{3}}-\frac{2 k \cosh \! \left(r \right) \chi'\! \left(\frac{t}{r}\right) t}{\sinh \! \left(r \right) r^{2}}\right) h_{BCD}\! \left(r , s , t\right)}\\
&
{-\frac{2 \chi'\! \left(\frac{t}{r}\right) t \left(\frac{\partial}{\partial r}h_{BCD}\! \left(r , s , t\right)\right)}{r^{2}}+\frac{2 \chi'\! \left(\frac{t}{r}\right) t \left(\frac{\partial}{\partial r}h_{A}\! \left(r , s , t\right)\right)}{r^{2}}.}
\end{align*}

Note that $\chi'(t/r)$ and $\chi''(t/r)$ are zero unless $1<t/r<2$ so that $r>t\geq (R^2+1)^{3/2}/2$ and that they are universally bounded by $M_1$ and $M_2$ respectively.  To show that $R(r,s,t)=O(1/t^2)$, it suffices then to examine $\frac{\partial}{\partial r}h_A^\pm(r,s,t)$ and $\frac{\partial}{\partial r}h_{BCD} (r,s,t)$.  It is not difficult to see using Lemma \ref{G} and Lemma \ref{crude} that both these functions are $O(1/t^2)$. 

\bigskip

\subsection{Application of the parabolic principle to the upper Region}

We now select $M>(R^2+1)^{3/2}$ for Region 0, say $M=(R^2+1)^{3/2}+1$.

Let $t=T_1=(R^2+1)^{3/2}+1/2$.  To fix things, we assume that the ``correct sign'' is non-negative \emph{i.e.} $D_{\Op}\,h_{ABCD}^+\geq 0$.

By the study of Region 0, we know that $h_0\geq   C h_{\Op}$ on the Region 0.

Now, $h_{ABCD}^+$ and $h_0$ are both equivalent to \eqref{estim:conjecture} for $t=T_1$.  Hence, there exists $C>0$ such that $C\,h_{ABCD}^+(r,s,T_1)-h_0(r,s,T_1)\geq 0$ and, therefore,  $C\,h_{ABCD}^+(r,s,T_1)-h_{\Op}(r,s,T_1)\geq 0$. By the parabolic minimum principle, we find that $C\,h_{ABCD}^+(r,s,t)\geq h_{\Op}(r,s,t)$ for all $t\geq T_1$.

Taking the ``correct sign'' to be non-positive, we can in the same manner, find $C>0$ such that $ h_{\Op}(r,s,t)\geq C\,h_{ABCD}^-(r,s,t)$ for all $t\geq T_1$.

This ends the proof of Theorem \ref{Main}.


\end{document}